\documentclass[a4paper]{amsart}
\usepackage{amssymb,cmap,verbatim}
\usepackage{stmaryrd} % to get the proper \bigtriangleup operator for diagonal intersection
\usepackage[colorlinks,citecolor=blue,urlcolor=black,linkcolor=black]{hyperref}

\usepackage{enumerate}  
\newtheorem*{thma}{Theorem~A}
\newtheorem*{thmb}{Theorem~B}
\newtheorem*{thmc}{Theorem~C}
\newtheorem*{thmd}{Theorem~D}
\newtheorem*{cora}{Corollary~A}
\newtheorem*{cord}{Corollary~D}

\newtheorem{thm}{Theorem}[section]
\newtheorem{fact}[thm]{Fact}
\newtheorem{lemma}[thm]{Lemma}
\newtheorem{cor}[thm]{Corollary}

\newtheorem{claim}{Claim}[thm]
\newtheorem{subclaim}{Subclaim}[claim]
 
\theoremstyle{definition}
\newtheorem{defn}[thm]{Definition}
\newtheorem{convention}[thm]{Convention}

\theoremstyle{remark}
\newtheorem{remark}[thm]{Remark}

\newcommand{\rud}{rud}

\DeclareMathOperator{\crit}{crit}
\DeclareMathOperator{\trcl}{trcl}
\DeclareMathOperator{\clps}{clps}

\DeclareMathOperator{\ord}{OR}
\DeclareMathOperator{\card}{Card}
\DeclareMathOperator{\cf}{cf}
\DeclareMathOperator{\otp}{otp}

\DeclareMathOperator{\reg}{Reg}
\DeclareMathOperator{\dom}{dom}
\DeclareMathOperator{\rng}{rng}

\DeclareMathOperator{\Dl}{Dl}%{\mathfrak{Dl}}
\DeclareMathOperator{\supp}{supp}

\renewcommand{\mid}{\mathrel{|}\allowbreak}
\renewcommand{\restriction}{\mathbin\upharpoonright}

\newcommand{\cN}{\mathcal{N}}
\newcommand{\cM}{\mathcal{M}}
\newcommand{\s}{\subseteq}

\newcommand{\forces}{\Vdash}

\newcommand*{\axiomfont}[1]{\textsf{\textup{#1}}}

\newcommand{\zf}{\axiomfont{ZF}}
\newcommand{\zfc}{\axiomfont{ZFC}}
\newcommand{\gch}{\axiomfont{GCH}}

\newcommand{\lcc}{\axiomfont{LCC}}
\title{On Local Club Condensation}

\author{Gabriel Fernandes}
\address{Department of Mathematics, Bar-Ilan University, Ramat-Gan 5290002, Israel.}
\urladdr{http://u.math.biu.ac.il/\textasciitilde zanettg}
\email{zanettg@macs.biu.ac.il}
\thanks{$^{*}$The author is funded by the European Research Council (grant
	agreement ERC-2018-StG 802756) as a postdoctoral fellow at Bar-Ilan
	University.}


\begin{document}
\begin{abstract} We obtain results on the condensation principle called local club condensation. We prove that in extender models an equivalence between the failure of local club condensation and subcompact cardinals holds. This gives a characterization of $\square_{\kappa}$ in terms of local club condensation in extender models.  Assuming $\gch$, given an interval of ordinals $I$ we verify that iterating the forcing defined by Holy-Welch-Wu, we can preserve $\gch$, cardinals and cofinalities and obtain a model where local club condensation holds for  every ordinal in $I$ modulo those ordinals which cardinality is a singular cardinal. 
	We prove that if $\kappa$ is a regular cardinal in an interval $I$, the above iteration provides enough condensation for the combinatorial principle $\Dl_{S}^{*}(\Pi^{1}_{2})$, and in particular $\diamondsuit(S)$,  to hold for any stationary $S \subseteq \kappa$.
\end{abstract}
\date{\today}

\maketitle

\section{Introduction}

%\begin{itemize}
%\item prove that condensation always holds at cardinals, and it does not mean that it condensates with respect to the predicate $E$. (notice that for points above $\kappa^{+}$ the transitive collapse has a different universe and not only the top-predicate)
%\end{itemize}

  \emph{Local club condensation} is a condensation principle that abstracts some of the condensation properties of  $L$, G\"odels constructible hierarchy. Local club condensation was first defined in \cite{FHl} and it is part of the outer model program which searches for forcing models that have $L$-like features.

  \begin{convention}
 	The class of ordinals is denoted by $\ord$. 
 	The transitive closure of a set $X$ is denoted by $\trcl(X)$,
 	and the Mostowski collapse of a structure $\mathfrak B$ is denoted by $\clps(\mathfrak B)$.
 \end{convention}

   In order to define condensation principles we define filtrations which is an abstraction of the stratfication $\langle L_{\alpha} \mid \alpha < \ord \rangle $
   of $L$.
  
  \begin{defn}
  	Given ordinals $\alpha < \beta$ we say that $\langle M_{\xi} \mid \alpha < \xi < \beta \rangle $ is a \emph{filtration} iff 
  	\begin{enumerate}
  		\item for all $\xi \in (\alpha,\beta)$, $M_{\xi}$ is transitive, $\xi \subseteq M_{\xi}$,
  		\item for all  $\xi \in (\alpha,\beta)$, $M_{\xi} \cap \ord = \xi $, 
  		\item for all  $\xi \in (\alpha,\beta)$, $ |M_{\xi}| \leq \max \{\aleph_0,|\xi|\}$,
  		\item if $\xi < \zeta$, then $M_{\xi} \subseteq M_{\zeta}$,
  		\item if $\xi$ is a limit ordinal, then $M_{\xi}=\bigcup_{\alpha < \xi} M_{\alpha}$. 
  	\end{enumerate}
  \end{defn}
  
 \begin{convention}\label{Union}
  Given a filtration $\langle M_{\xi} \mid \xi < \beta \rangle$, if $\beta$ is a limit ordinal we let $M_{\beta}:=\bigcup_{\gamma < \beta} M_{\gamma}$.
 \end{convention}
  
   The following is an abstract formulation of the Condensation lemma that holds for the constructible hirerarchy $\langle L_{\alpha} \mid \alpha \in \ord \rangle $:

  \begin{defn} Suppose that $\kappa$ and $\lambda$ are regular cardinals and that $ \vec{M} = \langle M_{\alpha} \mid \kappa < \alpha < \lambda \rangle $ is a filtration. We say that $\vec{M}$ satisfies \emph{strong condensation} iff for every $\alpha \in (\kappa,\lambda)$ and every $ (X,\in) \prec_{\Sigma_{1}} (M_{\alpha},\in)$ there exists $\bar{\alpha}$ such that $\text{clps}(X,\in) = (M_{\bar{\alpha}},\in)$.
  \end{defn}

   While strong condensation is not consistent with the existence of large cardinals, see \cite{FHl} and \cite{schvlck},   Local club condensation, which we define below, is consistent with any large cardinal, see \cite[Theorem1]{FHl}.

\begin{defn}[Holy,Welch,Wu,Friedman \cite{HWW},\cite{FHl}]	 \label{LCCupto}
	Let $ \kappa $ be a cardinal of uncountable cofinality.
	We say that $\vec{M}=\langle M_\beta \mid \beta < \kappa \rangle $ is a witnesses to the fact that \emph{local club condensation holds in $(\eta,\zeta)$},
	and denote this by $\langle H_{\kappa},{\in}, \vec M\rangle \models \lcc(\eta,\zeta)$,
	iff all of the following hold true:
	\begin{enumerate}
		\item $\eta < \zeta \leq \kappa+1$;
		\item $\vec M$ is a \emph{ filtration} such that $M_{\kappa}= H_\kappa$ \footnote{See Convention \ref{Union}},
		\item For every ordinal $\alpha$ in the interval $(\eta,\zeta)$ and every sequence $\mathcal{F} = \langle (F_{n},k_{n}) \mid n \in \omega \rangle$ such that, for all $n \in \omega$, $k_{n} \in \omega$ and $F_{n} \subseteq (M_{\alpha})^{k_{n}}$, there is a sequence 
		$\vec{\mathfrak{B}} = \langle \mathcal{B}_{\beta} \mid \beta < |\alpha| \rangle $ having the following properties: 
		\begin{enumerate}
			\item for all $\beta<|\alpha|$, $\mathcal{B}_{\beta}$ is of the form $\langle B_{\beta},{\in}, \vec{M} \restriction (B_{\beta} \cap\ord),  (F_n\cap(B_\beta)^{k_n})_{n\in\omega} \rangle$;
			\item for all $\beta<|\alpha|$, $\mathcal{B}_{\beta} \prec \langle M_{\alpha},{\in}, \vec{M}\restriction \alpha, (F_n)_{n\in\omega} \rangle$;\footnote{Note that the case $ \alpha= \kappa $ uses Convention~\ref{Union}.}
			\item for all $\beta<|\alpha|$, $\beta\s B_\beta$ and  $|B_{\beta}| < |\alpha|$;
			\item for all $\beta < |\alpha|$, there exists $\bar{\beta}<\kappa$ such that
			$$\clps(\langle B_{\beta},{\in}, \langle B_{\delta} \mid \delta \in B_{\beta}\cap\ord \rangle \rangle) = \langle M_{\bar{\beta}},{\in}, \langle M_{\delta} \mid \delta \in \bar{\beta} \rangle \rangle;$$
			\item $\langle B_\beta\mid\beta<|\alpha|\rangle$ is $\s$-increasing, continuous and converging to $M_\alpha$.
		\end{enumerate}
	\end{enumerate}
	For $\vec{\mathfrak{B}}$ as in Clause~(3) above we say that 
	\emph{$\vec{\mathfrak{B}}$ witnesses $\lcc(\eta,\zeta)$ at $\alpha$ with respect to $\mathcal{F}$}.
	We write $\lcc(\eta,\zeta]$ for $\lcc(\eta,\zeta+1)$.
\end{defn} 

 In section 2 we present our resutls regarding Local Club Condensation in extender models.
 An \emph{extender model} is an inner model of the form $L[E]$, it is a generalization of $L$ that can accommodate large  cardinals.  An inner model of the form $L[E]$ is the smallest transitive proper class that is a model of $\zf$ and is closed under the operator $x \mapsto x \cap E$, where $E:\ord \rightarrow V$ and each $E_{\alpha} = \emptyset $ or $E_{\alpha}$ is a partial extender. $L[E]$ models can  be stratified using the $L$-hirearchy and the $J$-hirearchy, for example: 
 \begin{itemize}
 	\item $J_{\empty}^{E}=\emptyset$,
 	\item $J_{\alpha+1}^{E}= \rud_{E}(J_{\alpha}^{E}\cup\{J_{\alpha}^{E}\})$,
 	\item $J_{\gamma}^{E}=\bigcup_{\beta < \gamma}J_{\beta}^{E}$ if $\gamma$ is a limit ordinal.
 \end{itemize}
 and finally $$L[E] = \bigcup_{\alpha \in \ord}J_{\alpha}^{E}.$$
 
 In \cite[Theorem 8]{FHl} it is shown that Local Club Condensation holds in various extender models,  we extend \cite[Theorem~8]{FHl} to an optimal result for extender models that are weakly iterable (see Defnition \ref{weaklyit}). We carachterize Local club condensation in extender models in terms of subcompact cardinals\footnote{A subcompact cardinal is a large cardinal that is located in the consistency strengh hirearchy below a supercompact cardinal and above a superstrong cardinal. See definition in \cite{SquareinK}}.

\begin{thma} \label{NoSubcompact}
	If $L[E]$ is an extender model that is weakly iterable, then given an infinite cardinal $\kappa$ the following are equivalent:
	\begin{itemize}
		\item[(a)] $\langle L_{\kappa^{+}}[E],{\in},\langle L_\beta[E]\mid\beta\in\kappa^{+} \rangle\rangle\models\lcc(\kappa^{+},\kappa^{++}]$.
		\item[(b)] 	$L[E] \models ( \kappa ~ \text{is not a subcompact cardinal})$. 
	\end{itemize}
	In addtion for every  limit cardinal $\kappa$ with $\cf(\kappa)>\omega$ we have \begin{center}$\langle L_{\kappa^{+}}[E],{\in},\langle L_\beta[E]\mid\beta\in\kappa^{+} \rangle\rangle\models\lcc(\kappa,\kappa^{+}].$\end{center}
	\end{thma}

We warn the reader that it is not known how to construct an extender model that is weakly iterable and has a subcompact cardinal, but this is part of the aim of the inner model theory program and it is desirable to know what hold in such models. 	
	
Corollary~A provides an equivalence between $\square_{\kappa}$ and a condensation principle that holds in the interval $(\kappa^{+},\kappa^{++})$, Corollary~ A is immediate from Theorem~A and the main result in \cite{MR2081183}:
	
\begin{cora}  If $L[E]$ is an extender model with Jensen's $\lambda$-indexing that is weakly iterable, then given $\kappa$, an $L[E]$ cardinal, the following are equivalent:
\begin{itemize}
	\item[(a)] $L[E]\models \square_{\kappa}$ 
	\item[(b)] $\langle L_{\kappa}[E], \in, \langle L_{\beta}[E] \mid \beta < \kappa^{+} \rangle \rangle \models \lcc(\kappa^{+},\kappa^{++})$
\end{itemize}	\end{cora}	

We verify that a subcompact cardinal is an even more severe impediment for $\lcc$ to hold:

\begin{thmb} Suppose $L[E]$ is an extender model with Jensen's $\lambda$-indexing such that every countable elementary submodel of $L[E]
	$ is $(\omega_{1}+1,\omega_{1})$-iterable. In $L[E]$, if an ordinal $\kappa$ is a subcompact cardinal, then there is no $\vec{M}$ such that $\langle M_{\kappa^{++}}, \in , \vec{M} \rangle \models \lcc(\kappa^{+},\kappa^{++})$ and $M_{\kappa^{++}}=H_{\kappa^{++}}^{L[E]}$ and $M_{\kappa^{+}}=H_{\kappa^{+}}^{L[E]}$.
\end{thmb}

 In section 3 we prove how to force local club condensation on a given interval of ordinals $I$ modulo ordinals with singular cardinality (Theorem~C) . It was already obtained in \cite{FHl} a model where local club condensation holds on arbritrary intervals $I$ including ordinals with singular cardinality, this was done via class forcing, although we do not obtain as much condensation as in \cite{FHl}, building on \cite{HWW} we define a set forcing $\mathbb{P}$ which is simpler than the forcing in \cite{FHl}, and which will force enough condensation for a few applications, see section 4.

%Let $\kappa$ be a regular cardinal and $\alpha$ an ordinal and let $\mathbb{S}$ denote all ordinals $\tau$ such that $|\tau|$ is a singular cardinal. In section 4 we will prove that assuming $\gch$, that iterating the set forcing in \cite{HWW} we can obtain their variant of local club condensation to hold for ordinals in the in the set $(\kappa,\kappa^{+\alpha}) \setminus \modulo \mathbb{S}$ while preserving $\gch$ and cofinalities. We will also verify that the filtration obtained in the generic extension is slow at every inaccessible cardinal in the interval $(\kappa,\kappa^{+})$. This provides a simpler of forcing enough condensation to obtain the results in \cite{FMR}. We believe our presentation of the iteration is simpler than adapting \cite{FHl} and may be useful for future applications of condensation principles.  

\begin{thmc} 
	If $\gch$ holds and $\kappa$ is a regular cardinal and $\alpha$ is an ordinal, then there is a set forcing $\mathbb{P}$ which is $<\kappa$-directed closed and $\kappa^{+\alpha+1}$-cc, $\gch$ preserving such that in $V^{\mathbb{P}}$ there is a filtration $\langle M_{\alpha} \mid \alpha < \kappa^{+\alpha+1} \rangle $ such that for every regular cardinal $ \theta \in [\kappa,\kappa^{+\alpha+1}]$ we have $H_{\theta}=M_{\theta}$ and $\langle M_{\alpha}\mid \alpha < \kappa^{+\alpha}\rangle \models \lcc_{\reg}(\kappa,\kappa^{+\alpha}) $.
\end{thmc}

In Section 4 we show that the iteration of the forcing from \cite{HWW} implies  $\Dl^{*}_{S}(\Pi^{1}_{2})$ (see definition in \cite{FMR}) which is a combinatorial principle defined in \cite{FMR} and is a variation of Devlin's $\diamondsuit^{\sharp}_{\kappa}$ (see \cite{Devlin}).

\begin{cord} Let $\kappa$ be an uncountable regular cardinal and let $\mu$ be a cardinal such that $\mu^{+} \leq \kappa$. If $\gch$ holds, then there is a set forcing $\mathbb{P}$ which is $<\mu^{+}$-directed and $\kappa^{+}$-cc, $\gch$ preserving, such that in $V^{\mathbb{P}}$ we have $\Dl^{*}_{S}(\Pi^{1}_{2})$, in particular  $\diamond(S)$, for any stationary $S \subseteq \kappa$.
\end{cord}
\section{Local club condensation in extender models}
  
  The main result in this section is Theorem~A which extends \cite{FHl}[Theorem~8] and gives a characterization of  local club condensation in terms of  subcompact cardinals. 
  For the standard notation on inner model theory and fine structure like \textit{premouse, projectum, standard parameter} and etc. we refer the reader to \cite{MR1876087}.

\begin{defn}
	Given a premouse $\mathcal{M}$, a parameter $p \in (\mathcal{M} \cap \ord^{<\omega})$ and $\xi \in \mathcal{M} \cap \ord$ and $\langle \varphi_{i} \mid i \in \omega \rangle $ a primitive recursive enumeration of all $\Sigma_{1}$ formulas in the premice language we define $$T_{p}^{M}(\xi)=\{ (a,i) \in (\xi^{<\omega}\times \omega) \mid M\models \varphi_{i}(a,p)\}.$$
\end{defn}

\begin{fact}\label{Phi} Given $\langle \varphi_{n} \mid n \in \omega \rangle $ a primitive recursive enumeration of  all $ \Sigma_1 $ formulas in the premice language, there exists a $\Sigma_{1}$-formula $\Phi(w,x,y)$ in the premice language such that for any premouse $\cM$ the following hold:
	\begin{itemize}
		\item If $n \in \omega$ is such that  $\varphi_{n}(x) = \exists y \phi_n(x,y)$ and $\phi_n$ is $\Sigma_0$, then for every $x \in \cM$ there exists  
		$ y_0 \in \cM$ such that $(\cM\models   \phi_n(x,y_0))$ iff there exists $y_1 \in \cM$ such that $ (\cM \models \Phi(n,x,y_1))$
		\item For every $n \in \omega$ and for every $x \in \cM$, if there are $y_{0},y_{1} \in \cM$ such that  $ (\cM \models \Phi(n,x,y_0) \wedge \Phi(n,x,y_1) ))$ then $y_0 = y_1$
	\end{itemize}
\end{fact}
\begin{defn}
	Let $\cM$ be a premouse we denote by $h_1^{\cM}$ the partial function from $\omega \times \cM	$ into $\cM$ defined by the formula $\Phi$ from Fact \ref{Phi}. Given $ X \subseteq \cM$ and $ p \in \cM$ we denote by $h_{1}^{\cM}[X,p]$ the set $h_{1}^{\cM}[(X\times\{p\})^{<\omega}]$.	
\end{defn}
\begin{fact} \label{Fact1}
	\begin{enumerate}
		\item 	Suppose $L[E]$ is an extender model. Let $\gamma$ be an ordinal such that $E_{\gamma}\neq\emptyset$, then there exists $g \in J_{\gamma+1}^{E}$ such that $g:\lambda(E_{\gamma})\rightarrow \gamma$ onto.
		\item  $\mathcal{P}(J_{\gamma}^{E})\cap J_{\gamma+1}^{E} = \Sigma_{\omega}(J_{\gamma}^{E})  $
	\end{enumerate}
\end{fact}

\begin{lemma} \label{Collapse}
	Suppose $L[E]$ is an extender model  and $\gamma$ is such that $E_{\gamma} \neq \emptyset$ and $L_{\gamma}[E]=J_{\gamma}^{E}$.  Then there exists $g \in L_{\gamma+1}[E]$ such that $g:\lambda(E_{\gamma})\rightarrow \gamma$ and $g$ is onto.
\end{lemma}
\begin{proof}  It follows from Fact \ref{Fact1} and that $\Sigma_{\omega}(J_{\gamma}^{E}) = L_{\gamma+1}[E]$.
	\end{proof} 

\begin{remark}
	Notice that in particular for any premouse $\cM$, if $\gamma \in \cM \cap \ord $ and $\cM\models ``\gamma \text{ is a cardinal}"$ it follows from Fact \ref{Fact1} that $ E_{\gamma} = \emptyset$, as otherwise $$J_{\gamma+1}^{E} \models ``\gamma \text{ is not a cardinal},"$$ and hence $$\cM\models  ``\gamma \text{ is not a cardinal}."$$ 
\end{remark}

\begin{defn}\label{weaklyit}
We say that an extender model $L[E]$ is \emph{weakly iterable} iff for every $\alpha \in \ord$ if there exists an elementary embedding $\pi:\bar{\cM} \rightarrow (J_{\alpha}^{E},\in,E|\alpha,E_{\alpha})$, then $\bar{\cM}$ is $(\omega_{1}+1,\omega_{1})$-iterable.\footnote{See definiton 9.1.10 in \cite{MR1876087} for the definition of $(\omega_{1}+1,\omega_{1})$-iterable. }	
\end{defn}

\begin{lemma} \label{Condensation} Let $L[E]$ be an extender model that is weakly iterable and let $\kappa$ be a cardinal in $L[E]$.
	Suppose $i:\mathcal{N} \rightarrow \mathcal{M}$ is the inverse of the Mostowisk collapse of $h_{1}^{\mathcal{M}}[\gamma \cup \{p_{1}^{\mathcal{M}}\}]$, $\rho_{1}(N)=\gamma$, $\crit(\pi)=\gamma$, $\gamma < \kappa$, $\mathcal{M} = \langle J_{\alpha}^{E}, \in, E\restriction \alpha, E_{\alpha} \rangle$ for some $\alpha \in (\kappa^{+},\kappa^{++})$.  Then  $\mathcal{N} \triangleleft \mathcal{M}$ if and only if $E_{\gamma}=\emptyset$. 
	\end{lemma}
	\begin{proof} The proof is a special case of condensation lemma. Suppose that  $E_{\gamma}= \emptyset$, we will verify that $\mathcal{N} \triangleleft \mathcal{M}$.
		 %$\mathcal{N}$ embedds into $\mathcal{M}$. % and hence, by lemma 9.2.7 in \cite{MR1876087}, the phalanx $\langle \mathcal{M}, \mathcal{N}, \gamma \rangle$ is countably iterable. % and then we can then compare $\langle \mathcal{M},\mathcal{N},\gamma \rangle$ and $\mathcal{M}$. Since $$\mathcal{M}\restriction \kappa^{+} \models (\kappa ~\text{is the largest cardinal}) ~ \& ~ (cf(\kappa)=\kappa) $$ 		 then $\mathcal{M}\restriction \kappa^{+}$ is universal stable mouse in the sense of \cite{MR3135495}. Hence $\mathcal{M}$ wins the comparison between $\langle \mathcal{M},\mathcal{N},\gamma \rangle $ and $\mathcal{M}$.
		
		 Let $H \prec_{\Sigma_{\omega}} V_{\Omega}$ for some $\Omega$ large enough, where $i \in H$ and $H$ is countable. Let $ \pi:\bar{H} \rightarrow V_{\Omega}$ be the inverse of the Mostowisk colapse of $H$, let $\pi(\bar{\mathcal{N}}) = \mathcal{N}$, $\pi(\bar{\mathcal{M}})=\mathcal{M}$ and $\pi(\bar{i})=i$.
		 
		 Let $e$ be an enumeration of $\bar{\cM}$ and let $ \Sigma$ be an $e$-minimal $(\omega_{1},\omega_{1}+1)$-strategy for $\bar{M}$ \footnote{The existence of an $e$-minimal iteration strategy follows from the hypothesis that $L[E]$ is weakly iterable and Neeman-Steel lemma, see \cite[Theorem9.2.11]{MR1876087}}. Since $\bar{i}$ embedds $\langle \bar{\cM},\bar{\cN},\bar{\gamma} \rangle$ into $\bar{\cM}$, it follows from 9.2.12 in \cite{MR1876087} that we can compare $\langle \bar{\cM},\bar{\cN},\bar{\gamma}\rangle$ and $\bar{\cM}$ and we have the following: 
		 \begin{itemize}
		 	\item $\bar{\cM}$ wins the comparison,
		 	\item the last model on the phalanx side is above $\bar{\cN}$,
		 	\item there is no drop on the branch of the phalanx side.
		 \end{itemize}
	
	 From the fact that $h_{1}^{\mathcal{N}}(\gamma \cup \bar{p}) = \mathcal{\cN}$ it follows that $h_{1}^{\bar{\mathcal{N}}}(\bar{\gamma}\cup q)=\bar{\mathcal{N}}$ where $\pi(q)=\bar{p}$. This implies that $\bar{\mathcal{N}}$ can not move in the comparison, as otherwise it would drop and we already know that it is the $\bar{\mathcal{M}}$ side which wins the comparison. Let $\mathcal{T}$ be the iteration tree on $\bar{\cM}$ and $\mathcal{U}$ the iteration tree on the phalanx $\langle \bar{M},\bar{N},\bar{\gamma} \rangle$.
	
	\begin{claim}
		$\bar{\mathcal{N}} \neq \mathcal{M}^{\mathcal{T}}_{\infty}$
	\end{claim}
	\begin{proof} We already know that $\bar{\mathcal{N}} \triangleleft \mathcal{M}^{\mathcal{T}}_{\infty}$. If $\bar{\mathcal{M}}$  does not move then $\bar{\mathcal{N}} \neq \mathcal{M}^{\mathcal{T}}_{\infty}=\mathcal{M}$ since they have different cardinality. Suppose $\mathcal{T}$ is non-trivial and $\mathcal{M}^{\mathcal{T}}_{\infty} = \bar{\mathcal{N}}$ let $b^{\mathcal{T}}$ be the main branch in $\mathcal{T}$. Let $\eta $ be the last drop in $b^{\mathcal{T}}$. In order to $\mathcal{M}^{\mathcal{T}}_{\infty}$ be 1-sound we need $\crit(E_{\eta}^{\mathcal{T}}) < \rho_{1}(\mathcal{M}_{\eta}^{\mathcal{T}})$ and since $\lambda(E_{0}^{\mathcal{T}}) > \gamma$ we have $\lambda(E_{\eta}^{\mathcal{T}}) > \gamma$. This implies that $\rho_{1}(\mathcal{M}_{\infty}^{\mathcal{T}}) \geq \rho_{1}(\mathcal{M}_{\eta+1}^{\mathcal{T}}) > \pi_{\eta^{*},\eta+1}^{\mathcal{T}}(\kappa_{\eta}) \geq \gamma = \rho_{1}(\mathcal{N})$, which is a contradiction since we are assuming that $\mathcal{M}^{\mathcal{T}}_{\infty} = \bar{\mathcal{N}}$. 
		\end{proof}
	
	Since $\bar{\mathcal{N}} $ is a proper initial segment of $\mathcal{M}^{\mathcal{T}}_{\infty}$ it will follows that $\bar{\mathcal{M}}$ does not move. For a contradiction, suppose $\bar{\mathcal{M}}$ moves then the index $\lambda(E_{0}^{\mathcal{U}}) $ of the first extender used on the $\bar{\mathcal{M}}$ side is greater than $\gamma$ since $\bar{\mathcal{N}}\restriction \gamma = \bar{\mathcal{M}}\restriction \gamma$ and, by our hypothesis, $E_{\gamma}= \emptyset$. Moreover the cardinal in $\mathcal{M}^{\mathcal{U}}_{lh(\mathcal{T}-1)}$, the last model in the iteration on the $\bar{\mathcal{M}}$ side of the comparison.  We have the following: \begin{itemize}
		\item $\bar{\mathcal{N}}$ is a proper initial segment of $\mathcal{M}^{\mathcal{T}}_{\infty}$,
		\item  $\lambda(E_{0}^{\mathcal{U}}) \leq (\bar{\cN}\cap \ord)$, 
		\item $h_{1}^{\bar{\mathcal{N}}}(\bar{\gamma}\cup q)=\bar{\mathcal{N}}$,
		\item $h_{1}^{\bar{\mathcal{N}}} \restriction (\bar{\gamma}\cup \{q\}) \in \mathcal{M}^{\mathcal{T}}_{\infty}$, 
	\end{itemize} then there exists a surjection from $\bar{\gamma}$ onto the index of $E_{0}^{\mathcal{T}}$ in $\mathcal{M}^{\mathcal{T}}_{\infty}$ which is a contradiction. 
	
	Thus we must have $\bar{\mathcal{N}} \triangleleft \bar{\mathcal{M}}$ and by elementarity of $\pi$ we have $\mathcal{N} \triangleleft \mathcal{M}$.	
		\end{proof}

\begin{lemma} \label{Club}
	Let $L[E]$ be an extender model that is weakly iterable. In $L[E]$, let $\kappa$ be a cardinal which is not a subcompact cardinal. Let $ \beta \in (\kappa^{+},\kappa^{++})$ and $\mathcal{M}= (J_{\beta}^{E},\in,E\restriction \beta, E_{\beta})$ and suppose that $\rho_{1}(\cM)= \kappa^{+}$. Then there is club $C \subseteq \kappa^{+}$ such that for all $\gamma \in C $ if $\mathcal{N} = \text{clps}(h_{1}^{\cM}(\gamma \cup \{p_{1}^{\cM}\}))$ then $\rho_{1}(\cN)= \gamma$. 
\end{lemma}	
\begin{proof} 
	Let $g$ be a function with domain $\kappa^{+}$ such that for each $\xi < \kappa^{+}$ we have that $g(\xi) = h_{1}^{\cM}(\xi \cup \{p_{1}^{\cM}\}) \cap \ord^{<\omega}$. 
	
	Let $f: \kappa^{+} \rightarrow \kappa^{+}$ where given $ \gamma < \kappa^{+} $,  $f(\gamma) $ is the least ordinal such that for every $r \in g(\xi)$ we have that $ T_{r}^{\cM}(\gamma) \in J_{f(\gamma)}^{E}$. Notice that $T_{r}^{\cM}(\gamma) \subseteq \bigcup_{n \in \omega} \mathcal{P}([\gamma]^{n})$, hence it can be codded as a subset of $\gamma$ and therefore, by acceptability, it follows that $f(\gamma) < \kappa^{+}$.
	
	 Let $ C$ be a club subset of the club of closure points of $f$ and such that $\gamma \in C $ implies $\gamma = h_{1}^{\cM}(\gamma \cup \{p_{1}^{\cM}\}) \cap \kappa^{+}$. We will verify that $C$ is the club we sought. 
	
	 Let $ \gamma \in C $. Let $\pi:\cN\rightarrow J_{\beta}^{E}$ be the inverse of the Mostowisk collapse of $h_{1}^{\cM}(\gamma \cup \{p_{1}^{\cM}\})$. Then for each $ \xi < \gamma $ we have $T_{r}^{\cN}(\xi)=T_{\pi(r)}^{\cM}(\xi) \in J_{\gamma}^{E} = \cN\restriction \gamma$. Therefore $\rho_{1}(\cN) \geq \gamma$. 
	 
	 Notice that by a standard diagonal argument  $a = \{ \xi \in \gamma \mid \xi \not\in h^{\cN_{\gamma}}_{1}(\xi,p_{1})\} \not\in \cN_{\gamma}$ since $\cN_{\gamma}=h_{1}[\gamma \cup \{p_{1}\}]$, thus $\rho_{1}^{\cM}\geq \gamma$.
	\end{proof}
	
\begin{lemma}\label{NoSubcompact} Let $L[E]$ be an extender model that is weakly iterable. Given $\kappa \in \ord$ if $\kappa$ is a successor cardinal in $L[E]$  then the following are equivalent:
	\begin{itemize}
		\item[(a)] $\langle L_{\kappa^{+}}[E],{\in},\langle L_\beta[E]\mid\beta\in\kappa^{+} \rangle\rangle\models\lcc(\kappa^{+},\kappa^{++}]$.
		\item[(b)] 	$L[E] \models ( \kappa ~ \text{is not a subcompact cardinal})$. 
	\end{itemize}
	and if $\kappa$ is a limit cardinal of uncountable cofinality, then  $${\langle L_{\kappa^{+}}[E],{\in},\langle L_\beta[E]\mid\beta\in\kappa^{+} \rangle\rangle\models\lcc(\kappa,\kappa^{+}]}.$$
	 \end{lemma}
\begin{proof}  Let $\alpha \in (\kappa^{+},\kappa^{++})$. Let $\beta \geq \alpha$ such that $\beta \in (\kappa^{+},\kappa^{++})$ and $\rho_{1}((J_{\beta}^{E},\in,E\restriction\beta,E_{\beta}) ) = \kappa^{+}$.  
    Let $\cM = (J_{\beta}^{E},\in,E\restriction\beta,E_{\beta})$, \[D= \{ \gamma <\kappa^{+} \mid h_{1}(\gamma \cup \{p_{1}\})\cap \kappa^{+}=\gamma \}\] and for each $\gamma \in D$ let $\cN_{\gamma} = \clps(h_{1}(\gamma \cup \{p_{1}\}))$. By Lemma \ref{Club} $D$ contains a club $ F \subseteq D$  such that  $\gamma \in F$ implies that there are $\pi_{\gamma}:N_{\gamma} \rightarrow J_{\beta}^{E}$ where $\pi_{\gamma}$ is $\Sigma^{(1)}_{1}$, $\pi_{\gamma}\restriction \gamma = id \restriction \gamma$, $\pi_{\gamma}(\gamma)  = \kappa^{+}$, $\rho_{1}(N_{\gamma}) = \gamma$. We can also assume that $\gamma \in F $ implies $L_{\gamma}[E]=J_{\gamma}^{E}$.  
    
    We verify first the implication $\neg(b) \Rightarrow \neg(a)$.
    
    Suppose that $\langle B_{\gamma} \mid \gamma < |\alpha| \rangle $ is a continuous chain of elementary submodels of $\cN = \langle J_{\alpha}^{E},\in,E|\alpha,E_{\alpha} \rangle$ such that for all $\gamma < |\alpha|$ we have $|B_{\gamma}| < |\alpha|$ and $\bigcup_{\gamma < |\alpha|} B_{\gamma} = \cM$. We will verify that for stationary many $\gamma$'s we have that $\clps(B_{\gamma}) $ is not of the form $J_{\zeta}^{E}$ for any $\zeta$.
    
    As $|\alpha|=\kappa$ is a regular cardinal it follows that for club many $\gamma$'s we have $B_{\gamma} = h_{1}^{\cM}(\gamma\cup\{p_{1}\}) \cap \cN = \pi^{-1}(L_{\alpha}[E])$ 
    
    From $\neg(b)$ by Schimmerling-Zeman carachterization of $\square_{\kappa}$ (see \cite{SquareinK}][Theorem~0.1]), we can assume that for stationary many $\gamma \in F$ we have $E_{\gamma}^{\cM} \neq \emptyset$ . Notice that $N_{\gamma} \models ``\gamma \text{ is a cardinal}"$ and therefore $E^{\cN_{\gamma}}=\emptyset$ by Proposition \ref{Collapse}.  On the other hand, from Proposition \ref{Collapse} we have  $L_{\gamma+1}^{E} \models ``\gamma \text{ is not a cardinal}"$. Since $L_{\gamma+1}[E]  \subseteq J_{\gamma+1}[E]$ it follows that  $\cN_{\gamma}  = \clps(B_{\gamma})$ is different from $J_{\zeta}^{E}$ for every $\zeta > \gamma$.  Therefore $\lcc(\kappa^{+},\kappa^{++})$ does not hold. 
    
    Next we verify $(b) \rightarrow (a)$. Suppose $\mathfrak{N} = \langle L_{\alpha}^{E},\in, E | \alpha, E_{\alpha}, (\mathcal{F}_{n} \mid n \in \omega  ) \rangle $. We can assume without loss of generality that $\beta$ is large enough so that $\mathfrak{N} \in \cM$. 
    
    %Hence from \ref{Condensation} we obtain $N_{\gamma} \triangleleft J_{\alpha}^{E}$. On the other hand for stationary many $\gamma \in \kappa^{+}$ we have that $E_{\gamma} \neq \emptyset$ which contradicts the fact that $N_{\gamma} \models \gamma $ is a cardinal, since by standard arguments $J_{\gamma+\omega}^{E} \models |\gamma| < \gamma$ for all $\gamma  \in \kappa^{+}$  such that $E_{\gamma} \neq \emptyset$. 

%$(b)\Rightarrow (a)$ Let $ \alpha $ such that $ \kappa^{+} < \alpha < \kappa^{++}$ and $L_{\alpha}[E]=J_{\alpha}^{E}$ \footnote{For pratical reasons at this point we need to work with such $\alpha$ in order to meet the definition of nice filtration, an alternative would be change the definition of nice filtration and index the nice filtration by limit ordinals the results on $\lcc$ would remain true.}

 We verify that $\langle L_{\tau}[E] \mid \tau < \kappa^{++} \rangle $  witnesses $\lcc(\kappa^{+},\kappa^{++}] $ at $\alpha$. 
Let $ \vec{\mathcal{R}} := \{ h_{1}^{\mathcal{M}}[\gamma \cup \{p_{1}^{\mathcal{M}} \}\cup \{u^{\mathcal{M}}_{1}\}] \mid \gamma < \kappa^{+} \}$ where for any given $X \subseteq \mathcal{M}$, $h_{1}^{\mathcal{M}}[X]$ denotes the $\Sigma_{1}$-Skolem hull of $X$ in $\mathcal{M}$ and $ p^{\mathcal{M}_{1}}$ is the first standard parameter. 
It follows that 
	$$C= \{ \gamma < \kappa^{+}  \mid \crit(\clps(h_{1}^{\mathcal{M}}[\gamma \cup \{p_{1}^{\mathcal{M}}, u^{\mathcal{M}}_{1}\}])) = \gamma\} $$ is a club
	and by lemma \ref{Club} 
	$$D=\{\gamma < \kappa^{+} \mid \rho_{1}(\mathcal{N}_{\gamma})=\gamma\} $$ is also a club. From Theorem 1 in \cite{MR1860606} and $(b)$  it follows that there is a club $F \subseteq \{ \gamma < \kappa^{+} \mid E_{\gamma}=\emptyset   \} $.
By Lemma  \ref{Condensation}, 	for every $\gamma \in D \cap F$ we have $N_{\gamma}\triangleleft \mathcal{M}$. We have $L_{\alpha}[E] \triangleleft \cM$, therefore $\clps(L_{\alpha}[E]\cap h_{1}^{\cM}(\gamma \cup \{p_{1}\})) \triangleleft \cN_{\gamma}$, hence $ \clps(L_{\alpha}[E]\cap h_{1}^{\cM}(\gamma \cup \{p_{1}\}))=\clps(B_{\gamma}) \triangleleft \cM$  which verifies the equivalence between $(b) \rightarrow (a)$

Now suppose $\kappa$ is a limit cardinal.  The same argument used for the implication $(b) \Rightarrow (a)$ follows  with the difference that we do not use Theorem 1 of \cite{MR1860606}, instead we use that the cardinals below $\kappa$ form a club and that for every cardinal $\mu< \kappa$ we have $E_{\mu}=\emptyset$. 
\end{proof} 

\begin{defn}
 Given two predicates $A$ and $E$ we say that $A$ is equivalent to $E$ iff $J_{\alpha}^{A}=J_{\alpha}^{E}$ for all $\alpha < \ord$. 
\end{defn}

\begin{cor}\label{IncompatiblePredicates}
If $A \subseteq \ord$ is such that \begin{itemize}
		\item $L[A] \models  (\kappa$ is a subcompact cardinal $),$ and
		\item $\langle L_{\kappa^{++}}[A],\in, \langle L_{\beta}[A] \mid \beta < \kappa^{++} \rangle \rangle \models \lcc(\kappa^{+},\kappa^{++}]$, 
	\end{itemize} then there is no extender sequence such that $L[E]$ is weakly iterable and $E$ is equivalent to $A$. 
\end{cor}

\begin{remark} In \cite{FHl} from the hypothesis that there is $\kappa$ a subcompact cardinal in $V$, it is obtained $A \subseteq \ord$ in a class generic extension which satisfies the hypothesis of corollary \ref{IncompatiblePredicates}.
\end{remark}

\begin{cor}
	\label{NoWitness} Suppose that $L[E]$ is a extender model with Jensen's $\lambda$-indexing and for every ordinal $\alpha$ the premouse $\mathcal{J}_{\alpha}^{E}$ is weakly iterable. If $\kappa$ is an ordinal such that $$L[E] \models \kappa \text{ is a subcompact cardinal,}$$ then for no $\vec{M}=\langle M_{\alpha} \mid \alpha < \kappa^{++} \rangle $ with $M_{\kappa^{+}}=H_{\kappa^{+}} $, $M_{\kappa^{++}}=H_{\kappa^{++}} $ and $\langle H_{\kappa^{++}},\in,\vec{M}\rangle \models \lcc(\kappa^{+},\kappa^{++})$. 
\end{cor}

%Motivated by the property of acceptability that holds in extender models, we define the following property for nice filtrations:

%\begin{defn} Given a nice filtration  $\vec{M} = \langle M_{\beta} \mid \beta < \kappa^{+} \rangle$ we say that $\vec{M}$ is slow below $\kappa$ iff there exists $\beta < \kappa$ such that, for every cardinal $\alpha$ with $\beta\le\alpha<\kappa$, $H_{\alpha} = M_{\alpha}$.\end{defn}

\begin{defn} We say that a nice filtration $\vec{M}$ for $H_{\kappa^{+}}$ strongly fails to condensate iff there is a stationary set $S \subseteq \kappa^{+}$ such that for any $\beta \in S $ and any continuous chain $\vec{B}$ of elementary submodels of $M_{\beta}$ there are stationary many points $\alpha$  where $B_{\alpha}$ does not condensate.
\end{defn}
\begin{lemma} If $\vec{M}$ is a filtration for $H_{\kappa^{+}}$ with $M_{\kappa} = H_{\kappa}$ that strongly fails  to condensate, then there is no filtration $\vec{N}$ of $H_{\kappa^{+}}$ with $N_{\kappa}= H_{\kappa}$ that witnesses $\lcc(\kappa,\kappa^{+})$.
\end{lemma}
\begin{proof} Let $\vec{N}$ be a filtration of $H_{\kappa^{+}}$ with $N_{\kappa}=H_{\kappa}$ and $N_{\kappa^{+}} = H_{\kappa^{+}}$. Then there is a club $D \subseteq \kappa^{+} $ where $ N_{\beta} = M_{\beta}$ for every $\beta \in D$. Let $ \beta\in S \cap D $, and let $\vec{\mathfrak{B}}=\langle \mathfrak{B}_{\tau} \mid \tau < |\beta|=\kappa \rangle $ be any chain of elementary submodels of $M_{\beta}=N_{\beta}$.  
\end{proof}

\begin{cor}
	Suppose $V$ is an extender model which is weakly iterable. If there exists $\kappa$ such that $L[E] \models ``\kappa \text{ is a subcompact cardinal}"$, then there is no sequence $\vec{M} = \langle M_{\alpha} \mid \alpha < \kappa^{+}\rangle$ in $L[E]$ such that $\langle M_{\kappa^{+}}, \in,  \vec{M} \rangle \models  \lcc(\kappa^{+},\kappa^{++})$. 
\end{cor}
\section{Forcing Local Club Condensation}

In \cite{FHl} it is shown, via class forcing, how to obtain a model of local club condensation for all ordinals above $\omega_{1}$.  Later a simpler forcing was presented in \cite{HWW} which forces condensation on an interval of the form $(\kappa,\kappa^{+})$ where $\kappa$ is a regular cardinal. In this section we show that iterating the forcing from \cite{HWW} and obtain a set forcing $\mathbb{P}$ which forces local club condensation on all ordinals of an interval  $(\kappa,\kappa^{+\alpha})$ modulo ordinals with singular cardinality. We will denote by $\lcc_{\reg}(\kappa,\kappa^{+\alpha})$ (see Definition \ref{DefLCCReg}) the property that local club condensation holds for all ordinals in the interval $(\kappa,\kappa^{+\alpha})$ modulo those which cardinality is a singular cardinal.

 Iterating the forcing from \cite{HWW} gives us a  set forcing which is relatively simpler than the class forcing from \cite{FHl} and $\lcc_{\reg}(\kappa,\kappa^{+\alpha})$ is enough condensation for applications where $\lcc(\kappa,\kappa^{+\alpha})$ was used before, see Section 4 for applications.%, for example in \cite{FMR} a diamond pinciple denoted by $\text{Dl}_S(\Pi^{1}_{2})$ is derived from $\lcc(\mu,\kappa^{+})$ where $\mu < \kappa$ and $\kappa$ is an inaccessible cardinal. 

\begin{defn} \label{Slow}
		Let $\kappa$ be a regular cardinal and $\alpha$ an ordinal such that $\kappa^{+\alpha}$ is a regular cardinal.  We say that $\Psi(\vec{M},\vec{\theta})$ holds iff $\vec{M}=\langle M_{\gamma} \mid \gamma < \kappa^{+\alpha} \rangle$ is a filtration and  for every regular cardinal $\theta \in (\kappa,\kappa^{+\alpha})$ we have $M_{\theta} = H_{\theta}$.\end{defn}

\begin{defn} \label{DefLCCReg}
	Let $\kappa$ be a regular cardinal and $\alpha$ an ordinal.  We say that $\lcc_{\reg}(\kappa,\kappa^{+\alpha})$ holds iff there is a filtration $\vec{M}=\langle M_{\gamma} \mid \gamma < \kappa^{+\alpha} \rangle$ such that $\langle M_{
	\kappa^{+\alpha}}, \langle M_{\gamma} \mid \gamma < \kappa^{+\alpha} \rangle \rangle \models \lcc(\alpha) $ for all $\alpha \in (\kappa,\kappa^{+\alpha})$ with $|\alpha|\in\reg$.
\end{defn}

The main result of this section is the following:

\begin{thmc}\label{ThmHWW}
Suppose $V$ models $\zfc + \gch$ and $\kappa$ is a regular cardinal and $\beta$ is an ordinal. Then there exists a set-sized forcing $\mathbb{P}$ which is cardinal preserving, cofinality preserving, $\gch$ preserving and forces the existence of a filtration $\vec{M}$ such that $\Psi(\vec{M},\kappa,\kappa^{+\beta})$ holds and  $\langle M_{\kappa^{+\beta}},\in,\vec{M}\rangle \models \lcc_{\reg}(\kappa,\kappa^{+\beta})$.
\end{thmc}

We start recalling the forcing from \cite{HWW} which we will iterate to obtain our model. We present the definitions of the forcing from \cite{HWW} for self containment, we will work mainly with abstract properties of the forcing from Theorem \ref{HWW} below.

\begin{convention}\begin{itemize}\item \label{IS} If $ a, b $ are sets of ordinals we write $ a\triangleleft b$ iff $ \sup(a) \cap b = a $. 
		\item If $ \mathbb{P} = \langle \langle \mathbb{P}_{\alpha} \mid \alpha < \beta \rangle, \langle \mathbb{\dot{Q}}_{\alpha} \mid \alpha + 1 < \beta \rangle \rangle $ is a forcing iteration, given $\zeta < \beta$ we denote by $\mathbb{\dot{R}}_{\zeta,\beta}$ a  $\mathbb{P}_{\alpha}$-name such that $ \mathbb{P} = \mathbb{P}_{\zeta}*\mathbb{\dot{R}}_{\zeta,\beta}$ (For the existence of such name $\dot{\mathbb{R}}_{\zeta,\beta}$ see for example \cite[Section 5]{MR823775}).
	\end{itemize}
	\end{convention}

\begin{defn} Let $\kappa$ be a regular cardinal. Suppose $\kappa \leq \alpha < \kappa^{+}$, a \emph{condition at $\alpha$} is a pair $(f_{\alpha},c_{\alpha})$ which is either trivial, i.e. $(f_{\alpha},c_{\alpha}) = (\emptyset,\emptyset)$, or there is $\gamma_{\alpha} < \kappa$ such that
\begin{enumerate}
	\item $c_{\alpha}:\gamma_{\alpha} \rightarrow 2$ is such that $C_{\alpha}((f_{\alpha},c_{\alpha}):=\{ \delta < \gamma_{\alpha} \mid c_{\alpha}(\delta)=1 \} = c_{\alpha}^{-1}\{1\}$.
	\item $f:\max(C_{\alpha}) \rightarrow \alpha$ is an injection and 
	\item $f_{\alpha}[\max(C_{\alpha})] \subseteq \max(C_{\alpha})$
\end{enumerate}
\end{defn}

\begin{defn}
Let $\kappa$ be a regular cardinal and $\alpha \in (\kappa,\kappa^{+})$ we also define a function $A$ with domain $[\kappa,\kappa^{+})$ such that for every $\alpha$, $A(\alpha)$ is a $\mathbb{H}_{\alpha}$-name for either $0$ or $1$. We fix a wellorder $\mathcal{W}$ of $H_{\kappa^{+}}$ of order-type $\kappa^{+}$. Let $\beta \in [\kappa,\kappa^{+})$ and assume that $A \restriction \beta$ and $\mathbb{H}_{\beta}$ have been defined. 

Let $A(\beta)$ be the canonical $\mathbb{H}_{\beta}$-name for either $0$ or $1$ such that for any $\mathbb{H}_{\beta}$-generic $G_{\beta}$, $A(\beta) = 1$ iff $ \beta = {\prec} \gamma, {\prec} \delta, \varepsilon {\succ} {\succ}$  \footnote{${\prec} \gamma, {\prec} \delta, \varepsilon {\succ} {\succ}$ denotes the G\"odel pairing.}, $\dot{x}$ is the $\gamma^{\text{th}}$-name (in the sense of $\mathcal{W}$) $\mathbb{H}_{\delta}$-nice name for a subset of $\kappa$, $\varepsilon < \kappa$ and $\varepsilon \in \dot{x}^{G_{\beta}}$ \footnote{As $\delta < \beta$, we identify $\dot{x}$ with a $\mathbb{H}_{\beta}$-name using the induction hypothesis that $\mathbb{H}_{\delta} \prec \mathbb{H}_{\beta}$.}

Suppose that $ A\restriction \beta $ is defined, we proceed to define $\mathbb{H}_{\beta}$. Suppose $p$ is an $\beta$-sequence such that for each $\alpha < \beta$ we have $p(\alpha) \in \mathbb{H}_{\alpha}$ and suppose that $|\supp(p)| = | \{\tau < \beta \mid p(\tau) \neq 1_{\mathbb{H}_{\tau}}\}| < \kappa$. If $\beta = \alpha +1$ for some $\alpha$, then we require that 
\begin{itemize}
		\item $p\restriction \beta \in \mathbb{H}_{\beta}$ for every $
		\beta < \alpha$ and if $\alpha = \beta + 1$, the following holds:
		\item $p(\beta)=(f_{\beta},c_{\beta})$ is a condition at $\beta$,
		\item if $C_{\beta}\neq \emptyset$, then $p\restriction \beta$ decides $A(\beta)=a_{\beta}$,
		\item for all $\delta \in C_{\beta}\neq \emptyset$,$p(\otp(f_\beta[\delta]))=a_\beta$,
		\item $\gamma^{p} = \supp(p)\cap \kappa = \gamma_\beta = dom(c_\beta)$ for any $\beta \in C\text{-}\supp(p)$, where $C\text{-}\supp(p):=\{\gamma < \beta \mid C_\gamma(p(\gamma)\neq \emptyset \}$
		\item  $\exists \delta^{p}$ such that for all $\beta \in C\text{-}\supp(p), \max(C_\beta)=\delta^{p}$,
		\item $\beta_0 <\beta_1 $ both in $C\text{-}\supp(p)$, $$f_{\beta_{0}}[\delta^{p}] \triangleleft f_{\beta_1}[\delta^{p}]$$
		\footnote{See Convention \ref{IS}} and 
		$$f_{\beta_1}[\delta^{p}] \setminus \beta_0 \neq \emptyset$$
		For $p$ and $q$ in $\mathbb{H}_{\alpha}$ we let $q\leq p $ iff $q \restriction \kappa \leq p \restriction \kappa$ and for every $\beta \in [\lambda,\alpha)
$, $q(\beta) \leq p(\beta)$. 
\end{itemize}
\end{defn}

We will work with a forcing that is equivalent to $\mathbb{H}_{\beta}$  and  is a subset of $ H_{\kappa^{+}}$. 

\begin{defn}\label{Forcingdefn}
	If $\kappa$ is a regular cardinal and $\pi:\mathbb{H}_{\kappa,\kappa^{+}} \rightarrow H_{\kappa^{+}}$ is such that $\pi(p) = p \restriction \supp(p)$, then we define $\mathbb{P}_{\kappa,\kappa^{+}}:= \rng(\pi)$ and given $s,t \in \mathbb{P}_{\kappa,\kappa^{+}}$ we let $s \leq_{\mathbb{P}_{\kappa,\kappa^{+}}} t $ iff $\pi^{-1}(s) \leq_{\mathbb{H}_{\kappa,\kappa^{+}}} \pi^{-1}(t)$.
\end{defn}

Next we describe how we will iterate the forcing from Definition \label{Forcingdefn}.

\begin{defn}
Let $\alpha$ be an ordinal and $\kappa$ a regular cardinal. We define $\mathbb{P}_{\kappa,\kappa^{+\alpha}} $ as the iteration $ \langle \langle \mathbb{P}_{\kappa,\kappa^{+\tau}} \mid \tau \leq \alpha \rangle , \langle \dot{\mathbb{Q}}_{\tau} \mid \tau < \alpha \rangle \rangle  $  as follows: 
\begin{enumerate}
	\item If $\tau = \beta+1$ for some $\beta < \tau$ and $\kappa^{+\beta}$ is a regular cardinal. If there exists $\dot{\mathbb{Q}}_{\beta}\subseteq H_{\kappa^{+\beta+1}}$  such that $\mathbb{P}_{\kappa,\kappa^{+\beta}} \forces \dot{\mathbb{Q}}_{\beta}=\mathbb{P}_{\kappa^{+\beta},\kappa^{+\beta+1}}$ we let $\mathbb{P}_{\kappa,\kappa^{+\beta+1}} = \mathbb{P}_{\kappa,\kappa^{+\beta}}*\dot{\mathbb{Q}}_{\beta}$, otherwise we stop the iteration.
	\item If $\tau = \beta+1$ for some $\beta < \tau$ and $\kappa^{+\beta}$ is a singular cardinal, we let $\dot{\mathbb{Q}}_{\beta} = \check{1}$.
	\item If $\beta$ is a limit ordinal and $\kappa^{+\beta}$ is a regular cardinal, then $\mathbb{P}_{\beta}$ is the direct limit of  $\langle \mathbb{P}_{\kappa,\kappa^{+\beta}} , \dot{\mathbb{Q}}_{\beta} 
	\mid \theta  < \tau \rangle $.% and $\mathbb{P}_{\tau} \forces \mathbb{Q}_{\beta}=\mathbb{P}_{\kappa^{+\beta}, \kappa^{+\beta+1}}$ where $\dot{\mathbb{Q}}_{\tau}\subseteq H_{\kappa^{+\beta+1}}$. In case there is no such $\dot{\mathbb{Q}}_{\beta}$ the construction stops,
	\item If $\tau$ is a limit ordinal and $\kappa^{+\beta}$ is singular, then $\mathbb{P}_{\kappa,\kappa^{+\tau}}$ is the inverse limit of  $\langle \mathbb{P}_{\kappa,\kappa^{+\theta}} , \dot{\mathbb{Q}}_{\theta} 
	\mid \theta  < \tau \rangle $.
	\end{enumerate}   
	\end{defn}

\begin{remark}
	Given an ordinal $\alpha$ and a regular cardinal $\kappa$ the forcing $\mathbb{P}_{\kappa,\kappa^{+\alpha}}$ is obtained  forcing $\mathbb{P}_{\kappa^{\beta},\kappa^{+\beta+1}} $ for each successor ordinal $\beta < \alpha$  such that $\kappa^{+\beta}$ is a regular cardinal. If $\beta$ is a limit ordinal but not an inaccessible cardinal, then we take inverse limits, if $\kappa$ is an inaccessible cardinal we take direct limits. 
\end{remark}

\begin{remark} \label{SeqBijections}
	Let $\kappa$ be a regular cardinal and let $G$ be $\mathbb{P}_{\kappa,\kappa^{+}}$-generic. Consider $$f_{\alpha}:= \bigcup \{f \mid \exists p ( \alpha \in \dom(p)\wedge p \in G \wedge p(\alpha)= (c,f) ) \}$$ By a standard density argument we have that $f_{\alpha}$ is a bijection from $\kappa$ onto $\alpha$.	It also holds that $\alpha \in A $ iff $\{\gamma < \kappa \mid \otp(f_{\alpha}[\gamma]) \in A\}$ contains a club and $\alpha \not\in A$ iff $\{\gamma < \kappa \mid \otp(f_{\alpha}[\gamma]) \not\in A\}$ contains a club.
\end{remark}

\begin{thm}[\cite{HWW}] \label{HWW}  Suppose $\gch$ holds and $\kappa$ is a cardinal. Then $\mathbb{P}_{\kappa}$ is a $<\kappa$- directed closed, $\kappa^{+}$-cc forcing such that $|\mathbb{P}|=\kappa$ and for all $G$, $\mathbb{P}$-generic, the following holds in $V[G]$:
	\begin{itemize} 
		\item There is $\vec{M} =\langle M_{\alpha} \mid \alpha \leq \kappa^{+} \rangle $ which witnesses $\lcc(\kappa,\kappa^{+})$, 
		\item $M_{\kappa}= H_{\kappa}$,
		\item  $M_{\kappa^{+}}=H_{\kappa^{+}}$,
		\item There exists $A \subseteq \kappa^{+}  $ such that for all $ \beta < \kappa^{+}$ we have $( M_{\beta}=L_{\beta}[A])$.
	\end{itemize}
\end{thm}

We will need the following facts: 

\begin{fact} \label{Baum} \cite[Theorem 2.7]{MR823775} Let $\mathbb{P}_{\alpha}$ be the inverse limit of $\langle \mathbb{P}_\beta, \dot{\mathbb{Q}}_{\beta} \mid \beta < \alpha \rangle$. Suppose that $\kappa$ is a regular cardinal and for all $\beta < \alpha$, $$\Vdash_{\mathbb{P}_{\beta}} \dot{\mathbb{Q}}_{\beta} \text{ is } \kappa\text{-directed closed}.$$
	Suppose also that all limits are inverse or direct and that if $\beta \leq \alpha$, $\beta $ is a limit ordinal and $\cf(\beta)<\kappa$, then $\mathbb{P}_{\beta}$ is the inverse limit of $\langle \mathbb{P}_{\gamma} \mid \gamma < \beta \rangle$. Then $\mathbb{P}_{\alpha}$ is $\kappa$-directed closed.  
\end{fact} 

\begin{fact} \label{cc} Suppose $\cf(\kappa)>\omega$. If $\mathbb{P}$ is a $\kappa$-cc forcing and $\mathbb{P} \forces \dot{\mathbb{Q}}$ is $\kappa$-cc, then $\mathbb{P}*\mathbb{Q}$ is $\kappa$-cc. 
\end{fact}

\begin{fact} \label{Htheta} Let $\mathbb{P}$ be a partial order and $\theta$ a regular cardinal. Suppose $\mathbb{P}$ preserves cardinals. If $G$ is $\mathbb{P}$-generic, then $H_{\theta}^{V} = H_{\theta}^{V[G]}$.
\end{fact}
\begin{proof} Let $G$ be $\mathbb{P}$-generic and $w \in H_{\theta}^{V[G]}$. Let $\delta = | \text{trcl}\{w\}|$ and suppose $f: \delta \rightarrow \text{trcl}\{w\}$ is a bijection. Let $ R \subseteq \delta $ such that $ (x,y) \in R $ if and only if $f(x) \in f(y)$. Then the Mostowski collapse of $(\delta,R)$ is equal to $w$. Since $\mathbb{P}$ is $< \theta$-closed, it follows that $(\delta,R) \in V$ and hence $ w \in V$. Since $\mathbb{P}$ preserves cardinals it follows that $|w|=\delta$ and $w \in H_{\theta}$.
	\end{proof}

\begin{fact}
	\label{ccHmu} Let $\mu$ be a cardinal. Suppose $\mathbb{P}$ is a forcing that is $\mu^{+}$-cc and $\mathbb{P}\subseteq H_{\mu^{+}}$. If $G$ is $\mathbb{P}$-generic, then $H_{\mu^{+}}^{V[G]}=H_{\mu^{+}}[G]$. 
\end{fact}
\begin{proof} We proceed by $\in$ induction on the elements of $H_{\mu^{+}}^{V[G]}$. Notice that it suffices to prove the result for subsets of $\kappa^{+}$, since every $ x \in H_{\kappa^{+}}^{V[G]}$ is of the form $\trcl(\gamma,R)$ for some $\gamma < \kappa^{+}$ and $R \subseteq \gamma\times \gamma$.  Let $x = \sigma[G] \in H_{\mu^{+}}^{V[G]}$ such that $x \subseteq \kappa^{+}$. As $\mathbb{P}$ is $\mu^{+}$-cc, there is an ordinal $\gamma$ such that $ x\subseteq \gamma $ and $1_{\mathbb{P}}\forces \sigma \subseteq \gamma$.   Let $\theta = \bigcup\{ A_{\tau}\times\{\tau\} \mid \exists p \in \mathbb{P} \exists \xi \in \gamma (p \forces \check{\xi} = \pi ) \wedge \tau   \in H_{\mu^{+}}   \}$ and each $A_{\tau} \subseteq \mathbb{P}$  such that:
	\begin{enumerate}
		\item $A_{\tau}$ is an antichain,
		\item $q \in A_{\tau}$ implies $ q \forces \tau \in \sigma $,
		\item $q \in A_{\tau}$ is maximal with respect to the above two properties.
	\end{enumerate}

It follows that $\theta \in H_{\mu^{+}}$ and $\theta[G]=\sigma[G]$. 
	\end{proof}

\begin{remark}
	If $\mathbb{P}$ and $\mu$  satisfy the hypothesis from Fact \ref{ccHmu} and $\sigma$ is a $\mathbb{P}$-name such that $1_{\mathbb{P}} \forces \sigma \subseteq H_{\mu^{+}}$, then using Fact \ref{ccHmu} we can find a $\mathbb{P}$-name $\pi \subseteq H_{\mu^{+}}$ such that $1_{\mathbb{P}} \forces \sigma = \pi$.
\end{remark}

\begin{remark}\label{Inacc}
Given  a regular cardinal $\kappa$ and a limit ordinal $\beta$, we have $\cf(\kappa^{+\beta}) = \cf(\beta)$. 
Therefore if $\beta < \kappa^{+\beta}$ it follows that $\kappa^{+\beta}$ is singular. On the other hand if $\kappa^{+\beta} = \cf(\kappa^{+\beta}) = \cf(\beta) \leq \beta$, then $\beta$ is a weakly inaccessible cardinal, i.e. a cardinal that is a limit cardinal and regular.  Thus $\kappa^{+\beta}$ is regular iff $\beta$ is a weakly inaccessible cardinal.
\end{remark}

\begin{convention} Let $\mathbb{P}$ be a set forcing and $\varphi(\sigma_0,\cdots,\sigma_n)$  a formula in the forcing language. We write $\mathbb{P} \forces \varphi(\sigma_0,\cdots, \sigma_n)$ if for all $p \in \mathbb{P}$ we have $p \forces \varphi(\sigma_0,\cdots,\sigma_n)$.
	\end{convention} 

\begin{lemma} Suppose $\gch $ holds. Let $\kappa$ be a regular cardinal and $\beta$ an ordinal. Then $\mathbb{P}_{\kappa,\kappa^{+\beta}}$ preserves  $\gch$, cardinals and cofinalities and if $\kappa^{+\beta}$ is a regular cardinal then there exists $\dot{\mathbb{Q}}_{\beta} \subseteq H_{\kappa^{+\beta+1}}$ a $\mathbb{P}_{\kappa,\kappa^{+\beta}}$-name such that $\mathbb{P}_{\kappa,\kappa^{+\beta}}\forces \mathbb{P}_{\kappa^{+\beta},\kappa^{+\beta+1}}=\dot{\mathbb{Q}}_{\beta}$.
	\end{lemma}
\begin{proof}  We prove the  lemma by induction. Besides the statement of the lemma we carry  the following additional induction hypothesis:
	\begin{enumerate}[$(1)_\beta$]
		\item $\mathbb{P}_{\kappa,\kappa^{+\beta}}$ preserves cardinals and cofinalities,
		\item If $\kappa^{+\beta}$ is a regular cardinal and not the successor of a singular cardinal, then $\mathbb{P}_{\kappa,\kappa^{+\beta}}$ is $\kappa^{+\beta}\text{-cc}$ and there exists $\dot{\mathbb{Q}}_{\beta} \subseteq H_{\kappa^{+\beta+1}}$ such that $\mathbb{P}_{\kappa,\kappa^{+\beta}} \forces (\mathbb{P}_{\kappa^{+\beta},\kappa^{+\beta+1}} = \dot{\mathbb{Q}}_{\beta}) $.
		\item If $\kappa^{+\beta}$ is a successor of a singular cardinal, then $\mathbb{P}_{\kappa,\kappa^{+\beta}}$ is $\kappa^{+\beta+1}$-cc,
		%\item If  $\tau < \beta$ and $\kappa^{+\beta}$ is a regular cardinal, then there exists $\dot{\mathbb{Q}}_{\beta} \subseteq H_{\kappa^{+\beta+1}}$ such that $\mathbb{P}_{\kappa,\kappa^{+\beta}}\forces (\mathbb{P}_{\kappa^{+\beta},\kappa^{+\beta+1}}=\dot{\mathbb{Q}}_{\beta})$.
		\item If $\kappa^{+\beta}$ is a singular cardinal, then $|\mathbb{P}_{\beta}| \leq \kappa^{+\beta+1}$
	\end{enumerate} 
	  For $\beta = 1 $ the lemma follows from Theorem \ref{HWW}.  Suppose that $(1)_{\theta}$ to $(4)_{\theta}$ and that the lemma holds for all $\theta < \beta$. We will verify that $(1)_{\beta}$ to $(4)_{\beta}$ and that the lemma holds for $\beta$. 
	  
	  	$\blacktriangleright $ Suppose $\beta = \theta + 1 $ for some ordinal $\theta$ such that $\kappa^{+\theta}$ is regular.  From $(2)_{\theta}$ in our induction hypothesis,  $\mathbb{P}_{\kappa,\kappa^{+\theta}}$ is $\kappa^{+\theta}$-cc, hence by Fact \ref{ccHmu}, for any  $G$, $\mathbb{P}_{\kappa,\kappa^{+\theta}}$-generic, we have  $H_{\kappa^{+\theta+1}}[G]= H_{\kappa^{+\theta+1}}^{V[G]}$. Thus there exists $\dot{\mathbb{Q}}_{\beta} \subseteq H_{\kappa^{\theta+1}}$ such that $\mathbb{P}_{\kappa,\kappa^{+\beta}} \forces ``\mathbb{P}_{\kappa^{\beta},\kappa^{+\beta+1}} = \dot{\mathbb{Q}}_{\beta}"$. 
	  	
	  	From our induction hypothesis  	$\mathbb{P}_{\kappa,\kappa^{+\theta}}$ preserves $\gch$, cardinals and cofinalities and from $(2)_{\theta}$ we have that $\mathbb{P}_{\kappa,\kappa^{+\theta}}$ is $\kappa^{+\theta}$-cc.  We also have that $$\mathbb{P}_{\kappa,\kappa^{+\theta}} \forces ``\mathbb{P}_{\kappa^{+\theta},\kappa^{+\theta+1}} \text{ preserves } \gch, \text{ cardinals, cofinalities and it is } \kappa^{+\theta+1}\text{-cc}".$$
 Altogether	implies that $\mathbb{P}_{\kappa,\kappa^{+\beta}}$, which is  $\mathbb{P}_{\kappa,\kappa^{+\theta}}*\dot{\mathbb{Q}}_{\beta}$, preserves $\gch$, cardinals and cofinalities. By Fact \ref{cc} we have that $\mathbb{P}_{\kappa,\kappa^{+\beta}}$ is $\kappa^{+\beta}$-cc.

	$\blacktriangleright $ Suppose $\kappa^{+\theta}$ is singular and $\beta = \theta + 1 $. By our induction hypothesis $(4)_{\theta}$ we have  $|\mathbb{P}_{\kappa^{+\theta}}| \leq \kappa^{+\theta+1}$. As $ \dot{\mathbb{Q}}_{\theta}$ is the trivial forcing, it follows that $|\mathbb{P}_{\kappa,\kappa^{+\beta}}|\leq \kappa^{+\beta}$ and $\mathbb{P}_{\kappa,\kappa^{+\beta}}$ is $\kappa^{+\beta +1}$-cc. Therefore if $G$ is $\mathbb{P}_{\kappa,\kappa^{+\beta+1}}$-generic, $H_{\kappa^{+\beta+1}}[G] = H_{\kappa^{+\beta+1}}^{V[G]}$, hence we can find $\dot{\mathbb{Q}}_{\beta+1} \subseteq H_{\kappa^{+\beta+1}}$ as sought and $\mathbb{P}_{\kappa,\kappa^{+\beta}}$ preserves $\gch$, cardinals and cofinalities. 
	
	$\blacktriangleright $ Suppose that $\beta$ is a limit ordinal and $\kappa^{+\beta}$ is a singular cardinal. 
		
	From our induction hypothesis we have that for every $\zeta < \beta $ the forcing $\mathbb{P}_{\kappa,\kappa^{+\zeta}}$ preserves cardinals and by Theorem \ref{Baum} $$\mathbb{P}_{\kappa,\kappa^{+\zeta}} \forces ``\dot{\mathbb{R}}_{\kappa^{+\zeta},\kappa^{+\beta}} \text{ is } {<}\kappa^{+\tau}\text{-closed}."$$ Therefore all cardinals below $\kappa^{+\beta}$ are preserved. Thus $\kappa^{+\beta}$ remains a cardinal in $V[G_{\beta}]$ and $\cf(\kappa^{+\beta})^{V[G_{\beta}]}=(\cf(\kappa^{+\beta}))^{V[G_\tau]}= (\cf(\kappa^{+\beta}))^V$.
		
		 As $ \cf(\kappa^{+\beta})^{+} < \kappa^{+\beta}$, we can fix $\tau < \beta$  such that $\kappa^{+\tau} \geq cf(\kappa^{+\beta})$. From our induction hypothesis we have that $\mathbb{P}_{\kappa,\kappa^{+\tau+1}}$ preserves cardinals. From Theorem \ref{Baum} we have that $\mathbb{P}_{\kappa,\kappa^{+\tau +1}}$ forces $\dot{\mathbb{R}_{\kappa^{+\tau+1},\kappa^{+\beta}}}$ to be ${<}\cf(\kappa^{+\beta})^{+}$-closed. Therefore $((\kappa^{+\beta})^{\cf(\kappa^{+\beta}}))^{V[G_{\tau}]} = ((\kappa^{+\beta})^{\cf(\kappa^{+\beta}}))^{V[G_{\beta}]} $ and $(\kappa^{+\beta+1})^{V[G_\theta]} = (\kappa^{+\beta+1})^{V[G_{\beta}]}$. We have verified above that
		 \begin{itemize}
		 	\item  $\mathbb{P}_{\kappa,\kappa^{+\beta}}\forces (\kappa^{+\beta})^{V} \text{is a cardinal}$
		 	\item $\mathbb{P}_{\kappa,\kappa^{+\beta}} \forces (\cf(\kappa^{+\beta}))=\cf^{V}(\kappa^{+\beta})$	
		 	\item 	 		 $\mathbb{P}_{\kappa,\kappa^{+\beta}}\forces 2^{\kappa^{+\beta}}=\kappa^{+\beta+1}$
		 \end{itemize}
	     It is also clear from the above that $\mathbb{P}_{\kappa,\kappa^{+\beta}}$ preserves $\gch$, cardinals and cofinalities below $\kappa^{+\beta}$.

		% For every $\zeta < \beta$ we have $|\mathbb{Q}_{\zeta+1}| \leq \kappa^{+\zeta+1} $, therefore $|\mathbb{P}_{\beta}| \leq \kappa^{+\beta+1}$ and  $\mathbb{P}_{\kappa,\kappa^{+\beta}}$ is $\kappa^{+\beta+2}$-cc. From our induction hypothesis for each $\zeta < \beta$ we have that $\mathbb{P}_{\kappa,\kappa^{+\zeta}}$ preserve $\gch$, cardinals and cofinalities and $\mathbb{P}_{\kappa,\kappa^{+\zeta}}$ is $\kappa^{+\zeta}$-cc or $\kappa^{+\zeta+1}$-cc depending on whether $\kappa^{+\zeta}$ is regular or singular respectively.  Since $\dot{\mathbb{R}}_{\kappa^{+\zeta},\kappa^{+\beta}}$ is $<\kappa^{+\zeta}$-closed, it follows that $\mathbb{P}_{\kappa,\kappa^{+\beta}}$ preserve $\gch$, cardinals and cofinalities up to $\kappa^{+\beta}$.
		   From our induction hypothesis $(2)_{\theta}$ it follows that for each $\theta < \beta$ we have $|\dot{\mathbb{Q}}_{\theta}|\leq \kappa^{+\theta+1}$, then using $\gch$ it follows that $|\mathbb{P}_{\kappa,\kappa^{+\beta}}| \leq \kappa^{+\beta+1}$ and hence $\mathbb{P}_{\kappa,\kappa^{+\beta}}$ is $\kappa^{+\beta+2}$-cc.
		 	
		 	Thus $\mathbb{P}_{\kappa,\kappa^{+\beta}}$ preserves $\gch$, cardinals and cofinalities above $\kappa^{+\beta+2}$.

		 $\blacktriangleright$ If $\kappa^{+\beta}$ is a limit cardinal and regular, then $\mathbb{P}_{\kappa,\kappa^{+\beta}}$ is the direct limit of $\langle \mathbb{P}_{\kappa,\kappa^{+\tau}},\dot{\mathbb{Q}}_{\tau+1} \mid \tau<\kappa^{+\beta} \rangle$. From our induction hypothesis $(2)_{\theta}$ for $\theta < \beta$, we have $|\dot{\mathbb{Q}}_{\tau}| \leq \kappa^{+\tau+1}$. Therefore $|\mathbb{P}_{\kappa^{+\beta}}| \leq \kappa^{+\beta}$ and hence $\mathbb{P}_{\kappa^{+\beta}}$ is $\kappa^{+\beta+1}$-cc  and preserves $\gch$, cardinals and cofinalties at cardinals greater or equal than $\kappa^{+\beta}$. From our induction hypothesis we have that cofinalities cardinals and $\gch$ are preserved  below $\kappa^{+\beta}$. Hence $\mathbb{P}_{\kappa,\kappa^{+\beta}}$ preserves cofinalities, cardinals and $\gch$.  
	\end{proof}

Lemma \ref{AbstractLimit}, below,  will be used in a context where $W_{\tau}= V[G_{\tau}]$ and $G_{\tau}$ is $\mathbb{P}_{\kappa,\kappa^{+\tau}}$-generic.

\begin{lemma}\label{AbstractLimit} Let $\langle W_{\tau} \mid \tau \leq \beta \rangle$ be a sequence of transitive proper classes that model $\zfc$ and suppose that $\tau_0 < \tau_1 <\beta $ implies $W_{\tau_0} \subseteq W_{\tau_{1}}$ and $\card^{W_{\tau_{0}}}=\card^{W_{\tau_{1}}}$. Suppose further that the following hold:
	\begin{enumerate} 
		\item for each $\tau < \beta$ the folowing holds in $W_{\tau}$: there exists $A_{\tau} \subseteq \kappa^{+\tau}  $ such that $ \vec{M}^{\tau} = \langle L[A_{\tau}]_{\zeta} \mid \zeta < \kappa^{+\tau + 1} \rangle$  witnesses $\lcc_{\reg}(\kappa,\kappa^{+\tau})$,
		\item For $\tau_0 < \tau_1 < \beta $ we have $H_{\tau_{0}^{+}}^{W_{\tau_0}}=L[A_{\tau_0}]_{\tau_{0}^{+}} = L[A_{\tau_1}]_{\tau_{1}^{+}}=H_{\tau_{1}^{+}}^{W_{\tau_1}}$, 
		\item for every $\tau < \beta$ we have $H_{\tau}^{W_{\tau}} = H_{\tau}^{W_{\beta}}$ and 
		\item   $$\mathbb{A}:=\bigcup\{ A_\tau 	\restriction (\kappa^{+\tau},\kappa^{+\tau+1}) \mid  \reg(\kappa^{+\tau}) \wedge \tau < \beta \} \cup \bigcup \{A_\tau 	\restriction (\kappa^{+\tau},\kappa^{+\tau+2}) \mid \text{Sing}(\kappa^{+\tau})\}$$ is an element of 
$ W_{\beta}$.
\end{enumerate}
 Then $\vec{M} = \langle L_{\zeta}[\mathbb{A}]
\mid \zeta < \kappa^{+\beta} \rangle $ witnesses $\lcc_{\reg}(\kappa,
\kappa^{+\beta})$ in $W_{\beta}$.	\end{lemma} 
\begin{proof} 	We work in $W_{\beta}$.  Let $\alpha \in (\kappa,\kappa^{+\beta})$ such that $|\alpha|$ is a regular cardinal. Let $ \mathbb{S}= \langle L_{\alpha}[\mathbb{A}], \in, (\mathcal{F}_{n})_{n\in\omega} \rangle  \in H_{|\alpha|^{+}}$. 
	
We will find $\vec{B}$ that witnesses $\lcc$ at $\alpha$ for $\mathbb{S}$.  There is $\vec{B_{0}} \in W_{\tau}$ where $\kappa^{+\tau} = |\alpha|$, which witnesses $\lcc$ at $\alpha$ in $W_{\tau}$ with respect to $(\mathcal{F}_{n})_{n \in \omega}$. Since $L_{\tau}[A_{\tau}]=H_{\tau}^{W_{\tau}} = H_{\tau}^{W_{\tau^{+}}}=L_{\tau}[A_{\tau^{+}}]$, it follows that there is a club $C \subseteq \kappa^{+\tau}$ such that $ \iota \in C $ implies $L_{\iota}[A_{\tau}] = L_{\iota}[\mathbb{A}]$. Thus $\vec{B}= \vec{B_0} \restriction C $ will witness $\lcc$ at $\alpha$ with respect to $(\mathcal{F}_n)_{n \in \omega}$ in $W_{\beta}$.
	\end{proof}

\begin{lemma}\label{Succ} Let $\kappa$ be a regular cardinal and $\beta$ an ordinal.  Suppose that there exists $\vec{M} = \langle L_{\alpha}[A] \mid \kappa \leq \alpha < \kappa^{+\beta} \rangle$ which witnesses $\lcc_{\reg}(\kappa,\kappa^{+\beta})$  and $\Psi(\vec{M},\kappa,\kappa^{+\beta})$ holds. If  $\kappa^{+\beta}$ is a regular cardinal, then $\mathbb{P}_{\kappa^{+\beta},\kappa^{+\beta+1}} \forces \lcc_{\reg}(\kappa,\kappa^{+\beta})$ and if $  \kappa^{+\beta}$ is a singular cardinal then $\mathbb{P}_{\kappa^{+\beta+1},\kappa^{+\beta+2}} \forces \lcc_{\reg}(\kappa,\kappa^{+\beta+1}) \wedge \Psi(\vec{M},\kappa,\kappa^{+\beta}) $.  
\end{lemma}
\begin{proof} We split the proof into two cases depending on whether $\kappa^{+\beta}$ is regular or not. 

	$\blacktriangleright$ Suppose $\kappa^{+\beta}$ is a regular cardinal.  Let $B \subseteq (\kappa^{\beta},\kappa^{\beta+1})$ such that $\langle L_{\alpha}[B] \mid \alpha < \kappa^{+\beta+1} \rangle $ witnesses $\lcc_{\reg}(\kappa^{\beta},\kappa^{+\beta+1})$. Since $\mathbb{P}_{\kappa^{+\beta},\kappa^{+\beta+1}} $ is $< \kappa^{+\beta}$-closed, it follows that for $G$, $\mathbb{P}_{\kappa^{+\beta},\kappa^{+\beta+1}}$-generic we have, by Fact \ref{Htheta} that   $(H_{\kappa^{+\beta}})^{V[G]} = (H_{\kappa^{+\beta}})^{V}$.     	
	We then let $\vec{N}= \langle L_{\alpha}[C]   \mid \alpha < \kappa^{+\beta+1} \rangle $ where $C = (A \cap \kappa^{+\beta}) \cup (B \setminus \kappa^{+\beta})$, witness $\lcc_{\reg}(\kappa,\kappa^{+\beta+1})$.  
    	
    	$\blacktriangleright$ Suppose $\kappa^{+\beta}$ is a singular cardinal. Let $G$ be $\mathbb{P}_{\kappa^{+\beta+1},\kappa^{+\beta+2}}$-generic over $V$. Let $G$ be $\mathbb{P}$-generic, from Fact \ref{Htheta} it follows that for every cardinal $\theta < \kappa^{+\beta+1}$ we have $H_{\theta}^{V} = H_{\theta}^{V[G]}$.
    	
    	 Let $ B \subseteq \kappa^{+\beta+2}$ be such that $\vec{N}= \langle L_{\gamma}[B] \mid \gamma < \kappa^{+\beta+2} \rangle$ witnesses $\lcc(\kappa^{+\beta+1},\kappa^{+\beta+2}) $ in $V[G]$. Let $ C:= A \cup (B \setminus \kappa^{+\beta})$. Then $\vec{W}:= \langle L_{\alpha}[C] \mid \alpha < \kappa^{+\beta+2} \rangle $ witnesses $\lcc_{\reg}(\kappa,\kappa^{+\beta+2})$.     
	\end{proof}

\begin{thmc} \label{succSing} 	If $\gch$ holds and $\kappa$ is a regular cardinal and $\alpha$ is an ordinal, then there is a set forcing $\mathbb{P}$ which is $<\kappa$-directed closed and $\kappa^{+\alpha+1}$-cc, $\gch$ preserving such that in $V^{\mathbb{P}}$ there is a filtration $\langle M_{\alpha} \mid \alpha < \kappa^{+\alpha} \rangle $ such that $\Psi(\vec{M},\kappa,\kappa^{+\beta+1})$ holds and $\langle M_{\alpha}\mid \alpha < \kappa^{+\alpha}\rangle \models \lcc_{\reg}(\kappa,\kappa^{+\alpha}) $
\end{thmc}
\begin{proof} We prove by induction  that the following hold:
	
		\begin{enumerate} 
		\item for each $\tau < \beta$  there exists $A_{\tau} \subseteq \kappa^{+\tau} \in V[G_{\tau}]$ such that, in $V[G_{\tau}]$ we have that  $ \vec{M}^{\tau} = \langle L[A_{\tau}]_{\zeta} \mid \zeta < \kappa^{+\tau + 1} \rangle \models \lcc_{\reg}(\kappa,\kappa^{+\tau})$ and $\Psi(\vec{M},\kappa,\kappa^{+\tau})$
		\item For $\tau_0 < \tau_1 < \beta $ we have $H_{\tau_{0}^{+}}^{V[G_{\tau}]}=L_{\tau_{0}^{+}}[A_{\tau_0}] = L_{\tau_{0}^{+}}[A_{\tau_1}]=H_{\tau_{1}^{+}}^{V[G_{\tau}] }$, 
		\item for every $\tau < \beta$ we have $H_{\tau}^{V[G_{\tau}]} = H_{\tau}^{V[G_{\beta}]}$ and 
		\item   $$\mathbb{A}:=\bigcup\{ A_\tau 	\restriction (\kappa^{+\tau},\kappa^{+\tau+1}) \mid  \reg(\kappa^{+\tau}) \wedge \tau < \beta \} \cup \bigcup \{A_\tau 	\restriction (\kappa^{+\tau},\kappa^{+\tau+2}) \mid \text{Sing}(\kappa^{+\tau})\}$$ is an element of 
		$ V[G_{\beta}]$.
	\end{enumerate}

	 If $\beta=1$ the lemma follows from Theorem \ref{HWW}. If $\beta = \theta+1$, from our induction hypothesis and Lemma \ref{Succ} it follows that $\mathbb{P}_{\kappa,\kappa^{+\beta+}} \forces \lcc_{\reg}(\kappa,\kappa^{+\beta})$. 
	If $\beta$ is a limit ordinal all we need to verify is that  \begin{equation}\label{eq1} \begin{gathered} H_{\kappa^{+\tau+1}}^{V_{\tau}} = H_{\kappa^{+\tau+1}}^{V_{\beta}} \end{gathered} \end{equation} for every $\tau < \beta$ in order to apply Lemma \ref{AbstractLimit}. Since for each $\tau < \beta$ we have that $\mathbb{P}_{\kappa,\kappa^{+\zeta}} \forces ``\dot{\mathbb{R}}_{\kappa^{+\zeta}} \text{ is }  <\kappa^{+\zeta}\text{-closed}"$ and $\mathbb{P}_{\kappa,\kappa^{+\zeta}}$ preserve cardinals and cofinalities \eqref{eq1} follows from Fact \ref{Htheta}

\begin{comment}Suppose that $\kappa^{\beta}$ is a limit cardinal. Let $B = \bigcup_{\theta < \beta} A_{\theta} \restriction (\kappa^{\theta},\kappa^{+\theta + 1})$, where $A_{\theta} \subseteq \kappa^{+\theta+1}$ is the predicate witnessing $\lcc(\kappa^{+\theta},\kappa^{\theta+1})$. We will prove that $\vec{M} = \langle L_{\alpha}[A] \mid \alpha < \theta \rangle  \models \lcc_{\reg}(\kappa, \kappa^{+\theta})$.

We have that $\mathbb{P}_{\kappa,\kappa^{\zeta}} \forces ``\dot{\mathbb{R}}_{\kappa^{\zeta}} \text{ is }  <\kappa^{+\zeta}\text{-closed}"$ and $\mathbb{P}_{\kappa,\kappa^{\zeta}}$ preserve cardinals and cofinalities. Therefore  for any $\beta < \theta $ such that $\kappa^{+\beta}$ is regular we have $(H_{\kappa^{+\beta}})^{V^{\mathbb{P}_{\kappa,\kappa^{+\theta}}}} = (H_{\kappa^{+\theta}})^{V^{\mathbb{P}_{\theta}}}$

$(H_{\kappa^{+\beta+1}})^{V^{\mathbb{P}_{\kappa,\kappa^{+\beta}}}} = (H_{\kappa^{+\beta+1}})^{V^{\mathbb{P}_{\beta+1}}}.$

\end{comment}

 	\end{proof}

\section{Applciations}

In this section we show that the iteration of the forcing from \cite{HWW} can replace some uses of the main forcing in \cite{FHl}.

%\begin{defn}	We say that $\Phi(\kappa,\alpha)$ holds iff there is a nice filtration $\vec{M}$ with $\dom(\vec{M})=\kappa^{+\alpha}$ that is slow at $\kappa^{+\beta}$ for all $	\beta \leq \alpha$ and $\kappa^{+\beta}$ is regular and witnesses $\lcc(\kappa,\kappa^{+\alpha})$.\end{defn}

\begin{comment}
\begin{lemma}\label{Delta1} Let $\kappa$ be a regular cardinal and $\alpha\in \ord$. Suppose $\gch$ holds in $V$. Let $\mu $ be a regular cardinal such that $\mu^{+}\leq \kappa$. Then $\mathbb{P}_{\mu,\kappa^{+}}$ forces that there exists $\vec{M}= \langle M_{\alpha} \mid \alpha < \kappa^{+} \rangle$, a filtration, such that \begin{enumerate}
		\item  $H_{\kappa}=M_{\kappa}$, $H_{\kappa^{+}}=M_{\kappa^{+}}$ 
		\item  there is $ A \subseteq 	\kappa^{+}$ such that for all $\alpha < \kappa^{+}$ we have $ M_{\alpha}= L_{\alpha}[A]$
		\item  $\langle M,\in,\vec{M} \rangle \models \lcc(\kappa,\kappa^{+})$
		there is a well order of $H_{\kappa^{+}}$ that is $\Delta_{1}$ on a parameter  $a \subseteq \kappa$	
	\end{enumerate} 
\end{lemma}
\begin{proof} Let $\vec{M}$ be given by Theorem~C. We know by Theorem~C that  (1) and (2) hold for $\vec{M}$.   Let us verify that (3) also holds.  Let $ \langle f_{\beta} \mid \kappa < \beta < \kappa^{+}  \rangle $ be the sequence of bijections obtained by forcing with $\mathbb{P}_{\kappa,\kappa^{+}}$. Then we have that $ \beta \in A$ iff $\{\gamma < \alpha \mid \otp(f_{\beta}[\delta]) \in A \}$ contains a club and $\beta \not\in A $ iff $ \{ \gamma < \alpha \mid  \otp(f_{\beta}[\gamma]) \not\in A \} $ contains a club.
	\end{proof}
\end{comment} 

\begin{defn} Let $\mu, A, \vec{F}$ be sets. We say that $\Xi(A,\mu,\vec{f})$ holds iff $\mu$ is a regular cardinal, $A $ is a function such that $A:\mu^{+}\rightarrow 2 $ and $\vec{f}$ is a sequence of bijections $\langle f_{\beta} \mid \kappa \leq \beta < \mu^{+} \rangle $ such that for each $\beta < \mu$, $f_{\beta}:\mu \rightarrow \beta$, and the following hold:
	\begin{itemize}
		\item $H_{\mu^{+}} = L_{\mu^{+}}[A]$,
		\item $ (\xi,1)\in A \setminus \mu \leftrightarrow \exists C ( C \text{ is a club } \wedge C \subseteq \{\gamma < \mu \mid \otp(f_\xi[\gamma])\in A \})$,
		\item $ (\xi,1) \in A  \setminus \mu \leftrightarrow \exists C ( C \text{ is a club } \wedge C \subseteq \{\gamma < \mu \mid \otp(f_\xi[\gamma])\in A \})$.
	\end{itemize}
%Let $\matchal{L} = \{\dot{\mu},\do{\xi},\dot{A},\in\}$, the language of set theory augmented by a 2-ary predicate $\dot{A}$. We denote by $T$ the theory $\zfc$ + $ \exists \ve c{f} \Xi(A,\mu,\vec{f})$.
\end{defn}
\begin{lemma}\label{complexity}  Let $\mu$ be a regular cardinal, $A $ a function $A:\mu^{+}\rightarrow 2 $ and $\vec{f}=\langle f_{\beta} \mid \kappa \leq \beta < \mu^{+} \rangle $ a sequence of bijections such that $f_{\beta}:\mu \rightarrow \beta$ for each $\beta < \mu^{+}$. Suppose $\Xi(A,\mu,\vec{f})$ holds. Given $\zeta \in \mu^{+}\setminus \mu$, the following are equivalent: 
 \begin{enumerate}
 	\item $\zeta \in A \setminus \mu $,
 	\item $\exists f  \exists C ( f:\mu \rightarrow \zeta \wedge f \text{ is a bijection } \wedge  C \text{ is a club } \wedge C \subseteq \{\gamma < \mu \mid \otp(f[\gamma])\in A \}$,
 	\item $ \forall f  \exists C ( f:\mu \rightarrow \zeta \wedge f \text{ is a bijection } \wedge  C \text{ is a club } \wedge C \subseteq \{\gamma < \mu \mid \otp(f[\gamma])\in A \}$,
 	\item $\forall f  \forall C ( f:\mu \rightarrow \zeta \wedge f \text{ is a bijection } \wedge  C \text{ is a club } \rightarrow C \not\subseteq \{\gamma < \mu \mid \otp(f[\gamma])\not\in A \}$.
 \end{enumerate} Moreover $\xi \in A \setminus \mu$ is $\Delta_1(\{A \restriction \mu,\xi\})$ over $H_{\mu^{+}}$.
\end{lemma}
\begin{proof} Let $\zeta \in \mu^{+}\setminus \mu$. As $\mu, A, \vec{f}$ witness the condensation axiom, it follows that 
	$\zeta \in A \setminus \mu \leftrightarrow \exists C ( C \text{ is a club } \wedge C \subseteq \{\gamma < \mu \mid \otp(f_\zeta[\gamma])\in A \}$. 
	Let $f$  be a bijection from $\mu$ onto $\zeta$. Then  from the regularity of $\mu$ it follows that there exists $ ( C \text{ a club }$ such that $D \subseteq \{\gamma < \mu \mid \otp(f_\zeta[\gamma])\in A \}$ iff there exists $D $ a club such that $ C \subseteq \{\gamma < \mu \mid \otp(f_\zeta[\gamma])\in A \}$. 
	Thus (1) (2) and (3) are equivalent. 
	
	Let us verify that (4) is equivalent to (1).  Since $ \mu,A,\vec{f}$ witness the condensation axiom, it follows that $ \zeta \not\in A $ iff  there exists a club $C$ such that $ C \subseteq \{ \gamma < \mu \mid \otp(f_{\zeta}[\gamma]) \not\in A \}$. Let $f$ be a bijection $f:\mu \rightarrow \zeta$. From the regularity of $\mu$ it follows that there exists a club $C$ such that $ C \subseteq \{ \gamma < \mu \mid \otp(f_{\zeta}[\gamma]) \not\in A \}$ iff there exists a club $D$ such that $ D \subseteq \{ \gamma < \mu \mid \otp(f_{\zeta}[\gamma]) \not\in A \}$. Thus (1) is equivalent to (4). 
	
	The moreover part follows from the equivalence between (1),(2) and (4), and the fact that $ C \not\subseteq  \{\gamma < \mu \mid \otp(f[\gamma])\not\in A \}$ is equivalent to $ \forall h \forall \gamma \forall \beta ( (\gamma \in C \wedge h:\beta \rightarrow f[\gamma] \wedge h \text{ is an isomorphis}) \rightarrow  (\beta,0) \in A)$.
	\end{proof}
Our next result, Theorem~D, is an adaptation of \cite[Theorem~39]{FHl}.

\begin{thmd}  Suppose that $\theta$ is an ordinal, $\kappa$ is a regular cardinal and $\kappa^{+\theta}$ is a regular cardinal. Then $\mathbb{P}_{\kappa,\kappa^{+\theta+1}}$ forces that $\lcc_{\reg}(\kappa,\kappa^{+\theta+1}) $ holds and that there exists a well order of $H_{\kappa^{+}}$ that is $\Delta_1$ definable over $H_{\kappa^{+}}$ in a parameter $a \subseteq \kappa^{+\theta}$. 
\end{thmd}
\begin{proof}  We have that $\mathbb{P}_{\kappa,\kappa^{+\theta+1}}$ forces that there exists $\vec{M}= \langle M_{\alpha} \mid \alpha < \kappa^{+\theta+1} \rangle$, a filtration, such that \begin{enumerate}
		\item  $H_{\kappa^{+\beta}}=M_{\kappa^{+\beta}}$ for every $\beta \leq \theta+1$,  
		\item  there exists $ A \subseteq 	\kappa^{+\theta+1}$ such that for all $\alpha < \kappa^{+\theta+1}$ we have $ M_{\alpha}= L_{\alpha}[A]$
		\item  $\langle M,\in,\vec{M} \rangle \models \lcc_{\reg}(\kappa,\kappa^{+\theta+1})$
		\end{enumerate} 
	 Let $ \langle f_{\beta} \mid \kappa < \beta < \kappa^{+}  \rangle $ be the sequence of bijections obtained by forcing with $\mathbb{P}_{\kappa,\kappa^{+}}$, see remark \ref{SeqBijections}. Then we have that $ \beta \in A$ iff $\{\gamma < \kappa^{+\theta} \mid \otp(f_{\beta}[\delta]) \in A \}$ contains a club and $\beta \not\in A $ iff $ \{ \gamma < \kappa^{+\theta} \mid  \otp(f_{\beta}[\gamma]) \not\in A \} $ contains a club. 
	 
	 Therefore by Lemma \ref{complexity} we can define $A \cap \kappa^{+\theta+1}$ in $H_{\kappa^{+\theta+1}}$ using $A \cap \kappa^{+\theta}$ with a $\Delta_1$ formula. The concatenation of the definition of $A$ with the $\Delta_1$ well order of $L_{\kappa^{+\theta+1}}[A]$ gives the $\Delta_1$ well order we sought.

\end{proof}

\begin{cord}  Suppose that $\theta$ is an ordinal, $\kappa$ is a regular cardinal. Then $\mathbb{P}_{\kappa,\kappa^{+\theta+1}}$ forces that for every $S \subseteq \kappa$ stationary we have $\Dl^{*}_{S}(\Pi^{1}_{2})$ and in particular $\diamondsuit(S)$.
	\end{cord}
\begin{proof} Follows from Theorem~D and \cite[Theorem~2.24]{FMR}.
	\end{proof}

\section{acknowledgments}

The author is greatful to Assaf Rinot and Miguel Moreno for several discussions on local club condensation. 
The author thanks Liuzhen Wu and Peter Holy for discussions on how to force local club condensation, and Farmer Schlutzenberg and Martin Zeman for discussions on condensation properties of extender models. 

\bibliographystyle{alpha}

\begin{thebibliography}{HWW15}
	
	\bibitem[Bau83]{MR823775}
	James~E. Baumgartner.
	\newblock Iterated forcing.
	\newblock In {\em Surveys in set theory}, volume~87 of {\em London Math. Soc.
		Lecture Note Ser.}, pages 1--59. Cambridge Univ. Press, Cambridge, 1983.
	
	\bibitem[Dev82]{Devlin}
	Keith~J. Devlin.
	\newblock The combinatorial principle {$\diamondsuit ^{\sharp }$}.
	\newblock {\em J. Symbolic Logic}, 47(4):888--899 (1983), 1982.
	
	\bibitem[FH11]{FHl}
	Sy-David Friedman and Peter Holy.
	\newblock Condensation and large cardinals.
	\newblock {\em Fundamenta Mathematicae}, 215(2):133--166, 2011.
	
	\bibitem[FMR20]{FMR}
	Gabriel Fernandes, Miguel Moreno, and Assaf Rinot.
	\newblock Inclusion modulo nonstationary.
	\newblock {\em Monatsh. Math.}, 192(4):827--851, 2020.
	
	\bibitem[HWW15]{HWW}
	Peter Holy, Philip Welch, and Liuzhen Wu.
	\newblock Local club condensation and {L}-likeness.
	\newblock {\em The Journal of Symbolic Logic}, 80(4):1361--1378, 2015.
	
	\bibitem[SV04]{schvlck}
	Ernest Schimmerling and Boban Velickovic.
	\newblock Collapsing functions.
	\newblock {\em Mathematical Logic Quarterly: Mathematical Logic Quarterly},
	50(1):3--8, 2004.
	
	\bibitem[SZ01a]{SquareinK}
	Ernest Schimmerling and Martin Zeman.
	\newblock Square in core models.
	\newblock {\em Bulletin of Symbolic Logic}, pages 305--314, 2001.
	
	\bibitem[SZ01b]{MR1860606}
	Ernest Schimmerling and Martin Zeman.
	\newblock Square in core models.
	\newblock {\em Bull. Symbolic Logic}, 7(3):305--314, 2001.
	
	\bibitem[SZ04]{MR2081183}
	Ernest Schimmerling and Martin Zeman.
	\newblock Characterization of {$\square_\kappa$} in core models.
	\newblock {\em J. Math. Log.}, 4(1):1--72, 2004.
	
	\bibitem[Zem02]{MR1876087}
	Martin Zeman.
	\newblock {\em Inner models and large cardinals}, volume~5 of {\em De Gruyter
		Series in Logic and its Applications}.
	\newblock Walter de Gruyter \& Co., Berlin, 2002.
	
\end{thebibliography}

\end{document}